\documentclass{amsart}

\usepackage{amssymb}
\usepackage{amsmath}
\usepackage{amscd}
\usepackage{amsthm}
\usepackage{graphicx}
\usepackage[all]{xy}

\usepackage{color}          
\usepackage{epsfig}

\usepackage{lineno}

\newcommand{\C} {\mathbb C}

\newtheorem{theorem}{Theorem}

\newtheorem{corollary}[theorem]{Corollary}
\newtheorem{proposition}[theorem]{Proposition}
\newtheorem*{definition}{Definition}

\title[On hyperbolic metric in holomorphic dynamics]{On hyperbolic metric and 
asymptotically finite invariant differentials 
in holomorphic dynamics.}

\author{Carlos Cabrera}
\address{Unidad Cuernavaca del Instituto de Matem\'aticas, 
UNAM.}
\email{carloscabrerao@im.unam.mx}
 \author{Peter Makienko }

\email{makienko@matcuer.unam.mx}

\begin{document}

 \begin{abstract}
Given a rational map $R$, we consider  the complement of the postcritical set 
$S_R$. In this paper we discuss the existence of invariant Beltrami 
differentials supported on a $R$ invariant subset $A$ of $S_R$.  Under some 
geometrical restrictions, either on the hyperbolic geometry of $A$ or on the 
asymptotic behavior of infinitesimal geodesics of the Teichm\"uller space of 
$S_R$, we show  the absence of invariant Beltrami differentials supported on 
$A$. 
In particular,  we show that if $A$ has finite 
hyperbolic area, then $A$ can not support invariant Beltrami differentials 
except in the case where $R$ is a Latt\`es map.  
\end{abstract}
\maketitle
\footnotetext{This work was partially supported by PAPIIT IN102515.}

\section{Introduction}

This article is a complementary part to the work done in \cite{CMFixed} with 
its own independent interest. We discuss geometric conditions under which 
there are no invariant Beltrami differentials supported on the dissipative set 
of a rational map $R$. 

In this paper we will always assume that the conservative set of the action of 
$R$ belongs to the Julia set.

Now, let us introduce the geometric objects to be treated in this 
paper. 

We denote by $P(R)$ the closure of the postcritical set of $R$ and consider 
the surface $S_R:=\bar{\C}\setminus P(R)$.  The surface $S_R$ is not always 
connected, however, on each connected component of $S_R$ we fix a 
Poincar\'e hyperbolic metric  and denote by $\lambda$ the family of all these 
metrics.

Let $Q(S_R)$ be the subspace of $ L_1(S_R)$ of holomorphic 
integrable functions on $S_R.$

A rational map $R$ defines a complex Push-Forward map on $L_1(\C)$, with 
respect to the Lebesgue measure $m$, which is a contracting endomorphism and 
is called the complex Ruelle-Perron-Frobenius, for shortness  
Ruelle operator. The Ruelle operator has the following formula:

\[R^*(\phi)(z)=\sum_{y \in R^{-1}(z)} \frac{\phi(y)}{R'(y)^2}
R(\zeta)\]
\[=\sum_i \phi(\zeta_i)(z)\zeta'_i(z)\] where $\zeta_i$ is any local complete 
system of branches of $R^{-1}.$ The space $Q(S_R)$ is invariant under the action 
of the Ruelle operator. 
The Beltrami operator $Bel:L_\infty(\C)\rightarrow L_\infty(\C)$ given by 
\[Bel(\mu)=\mu(R)\frac{\overline{R'}}{R'}\] is dual to the Ruelle operator 
acting on $L_1(\C)$. 

The fixed point space $Fix(B)$ of the Beltrami operator is called the 
\textit{space of invariant Beltrami differentials}. An element $\alpha \in 
L_\infty(\C)$ is called non trivial if and only if the functional given by 
\[v_\alpha(\phi) = \int \phi \alpha\]  is non zero on $Q(S_R).$ The norm of 
$v_\alpha$ in $Q^*(S_R)$, for a non trivial element $\alpha$,  is called the 
\textit{Teichm\"uller norm} of $\alpha$ and it is denoted by $\| \alpha 
\|_{T}.$ 

A non trivial element $\alpha$ is called \textit{extremal} if and only if the 
$\|\alpha\|_\infty=\|\alpha\|_T.$ 

A sequence of unit vectors $\{\phi_i\}$ is called a 
\textit{Hamilton-Krushkal} sequence, for short HK-sequence, for an extremal  
element $\alpha$ if and only if \[ \lim_{i\rightarrow 
\infty}|v_\alpha(\phi_i)|=\|\alpha\|_\infty.\]

A HK sequence $\{\phi_i\}$ is called \textit{degenerated} if converge to $0$ 
uniformly on compact sets.
 
Let $T:\mathcal{B}\rightarrow \mathcal{B}$ be a linear contraction of a Banach 
space $\mathcal{B}$. An element 
$b\in \mathcal{B}$ is called \textit{mean ergodic} with respect to $T$ if and 
only if the sequence of Ces\`aro averages with respect to $T$, given by
$C_n(b)=\frac{1}{n}\sum_{i=0}^{n-1} T^i(b)$, forms a weakly precompact family. 
Indeed (see Krengel \cite{Krengel}), when $\mathcal{B}$ is weakly complete then, for 
a mean ergodic element $b$,  the sequence $C_n(b)$ converges in norm to its limit, 
this limit always is  a fixed element of $T$. If every element $b\in \mathcal{B}$ 
is mean ergodic with respect to $T$ then the operator $T$ is called 
mean-ergodic.

By the Bers Representation Theorem, the space $Q^*(S_R)$ is linearly 
quasi-isome\-trically isomorphic to the \textit{Bergman} space $B(S_R)$ which 
is 
the space of holomorphic functions $\phi$ on $S_R$ with the 
norm $\|\lambda^{-2}\phi\|_{L_\infty(S_R)}.$

In the case where $S_R$ has finitely many components, a classical theorem, 
see for example \cite{Matsuzaki} and references within,  states that 
$Q(S_R)\subset B(S_R)$ if and only if the infimum of the length of simple 
closed geodesics is bounded away from $0$.

\section{Main Theorem}
Let $X$ be an $R$ invariant measurable set, then the set $W:=\bigcup 
R^{-n}(X)$ is completely invariant. In the following theorem we will only 
consider Ces\`aro averages with respect to the Ruelle operator $R^*$ in 
$L_1(W).$

\begin{theorem}\label{MainTechnicalThm}
Let $X$ be an $R$ invariant measurable subset such that the restriction map 
$r(\phi)= \phi|_X$ from $Q(S_R)$  to $L_1(X) $ is weakly precompact.  Then 
every $\phi \in Q(S_R)$ is mean ergodic with respect to $R^*$ in  $L_1(W)$. 
\end{theorem}

\begin{proof}
If $X$ is $R$ invariant then the Ruelle operator $R^*$ defines an endomorphism 
of $L_1(X)$. Given $\phi \in Q(S_R)$, the family of Ces\`aro averages 
$C_n(\phi)$ restricted on $X$ forms a weakly precompact subset of $L_1(X).$ 
We claim that $C_n(\phi)$ converges in norm  on $L_1(X).$  Indeed, first 
we show that every weak accumulation point  of $C_n(\phi)$ is a fixed point for 
the Ruelle operator. Let $f$ be  the weak limit of $C_{n_i}(\phi)$ for some 
subsequence $\{n_i\},$ then $R^*(f)$ is the weak limit of $R^*(C_{n_i}(\phi))$.
By the Fatou Lemma 
$$\int_X |f-R^*(f)| \leq \liminf \int_X |C_{n_i}(\phi)-R^*(C_{n_i}(\phi))|$$
$$\leq \liminf \| C_{n_i}(\phi)-R^*(C_{n_i}(\phi))\|_{L_1(S_R)}$$ 
$$\leq \limsup \|C_{n_i}(I-R^*)(\phi)\|_{L_1(S_R)}.$$

But 
$$ \|C_{n_i}(I-R^*)(\phi)\|_{L_1(S_R)}\leq \frac{2}{n_i}\|\phi\|_{L_1(S_R)}.$$
Then $f$ is a non zero fixed point of Ruelle operator. As in \cite{MakRuelle} 
we have that $|f|$ defines a finite absolutely continuous invariant measure. 
Hence, the support of $f$ is a non trivial subset of the conservative set of 
$R.$ By Lyubich's Ergodicity theorem (see \cite{Mc1} and \cite{LyuTypical}) and 
the fact 
that $X$ does not intersect the postcritical set we have $X=W=S_R$. But, 
McMullen's Theorem (Theorem 3.9 of \cite{Mc1}) implies that in this case $R$ is 
a, so called, \textit{flexible Latt\`es} map. Furthermore, 
the space $Q(S_R)$ is finitely dimensional and hence $R^*$ is a compact 
endomorphism of $Q(S_R)$, it follows that $R^*$ is mean ergodic on $Q(S_R)$. 

Therefore, if $R$ is not a flexible Latt\`es map then any weak limit of 
$C_n(\phi)$ is $0$. Since the weak closure of convex bounded sets is equal to 
the closure in norm of convex bounded sets, we conclude our claim.

Now let $W_n=R^{-n}(X)$, one can inductively prove that  
$\phi|_{W_n}$ is mean ergodic on $L_1(W_n)$. Indeed, let $\psi_n=\phi|_{W_n}$, 
since $R^*:L_1(W_n)\rightarrow L_1(W_{n-1}) \subset L(W_n)$ 
and $R^*(\psi_n)=R^*(\phi)|_{W_{n-1}}$, then by arguments above we are done. 

Now consider $\phi|_{W}-\phi|_{W_n}$, the $L_1$ norm of this difference 
converges to $0$ in $L_1(W)$, since the Ces\`aro averages does not expand the 
$L_1$ norm we have $$\|C_k(\phi|_{W}-\phi|_{W_n})\| \leq 
\|\phi|_{W}-\phi|_{W_n}\|.$$

Hence $C_k(\phi|_{W})$ converges to $0$ and $\phi$ is mean ergodic on $L_1(W)$. 

\end{proof}

Now we state our Main Theorem.

\begin{theorem}\label{MainTheorem}
Let $R$ be a rational map and let $X\subset S_R$ be an invariant 
measurable set of positive Lebesgue measure.
Assume that the restriction map $r(\phi)=\phi|_X$ from $Q(S_R)$ into 
$L_1(X)$ is weakly precompact. If $\mu$ is a  non trivial invariant Beltrami 
differential, then $m(supp (\mu)\cap X)>0$ if and only if $R$ is a flexible 
Latt\`es map.
\end{theorem}

\begin{proof}
Assume that $R$ is a flexible Latt\`es map. Then $R$ is ergodic on the Riemann 
sphere and therefore the support of any invariant Beltrami differential $\mu$ 
is the whole Riemann sphere. Hence, if $X$ is invariant of positive Lebesgue 
measure then $m(supp(\mu)\cap X)=m(X)>0.$ 

Again let $W=\bigcup R^{-n}(X)$. Now let $\mu$ be a non 
trivial invariant Beltrami differential supported on $W$. If $R$ is not 
Latt\`es, then for any $\phi \in Q(S_R)$  we have $$\int_{S_R} \phi \mu 
=\int_{S_R} \mu C_k(\phi)=\int_{W} \mu C_k(\phi).$$ 

By Theorem \ref{MainTechnicalThm}, the right hand side converges to $0$ as $k$ 
converges to $\infty$. Hence  $\int \phi \mu=0$ for every quadratic 
differential $\phi$ and the functional $\phi\mapsto \int \phi \mu$ is $0$ on 
$Q(S_R).$ Which contradicts the assumption that $\mu$ is non trivial. 

\end{proof}

In the proofs of the previous theorems, the only ingredient was the
precompactness of the Ces\`aro averages $C_n(\phi)$. Hence, 
it is enough to assume the weak precompactness only of 
Ces\`aro averages on elements of $Q(S_R)$. 
By results of the second author in \cite{MakRuelle}, see also a related work 
on \cite{CMFixed}, it  is enough to consider the Ces\`aro averages  of 
rational functions in $Q(S_R)$ having poles only on the set of critical values. 

\section{Compactness}
We want to discuss conditions under which the restriction map $\phi\mapsto 
\phi|_A$ is weakly precompact. Unfortunately, so far we have not found 
conditions where the restriction is weakly precompact but not compact.
Let us start with the following observations and definitions. 

\begin{definition}
A rational map $R$ satisfies the $B$-condition if and only if for any $\phi\in 
Q(S_R)$ we have $$\|\lambda^{-2}(z) 
\phi(z)\|_{L_\infty(S_R)}\leq C \|\phi(z)\|_{L_1(S_R)},$$ where $C$ 
is a constant independent of $\phi.$  
\end{definition}
In other words, if $R$ satisfies the $B$-condition, then
$Q(S_R)\subset B(S_R)$  and the inclusion
map $Q(S_R)\rightarrow B(S_R)$ is continuous. As it was noted on the 
introduction, this happens when $S_R$ has finitely many components and 
the infimum of the length of the simple closed geodesics is bounded away from 
$0.$

\begin{proposition}\label{prop.Bcond.comp}
If $R$ satisfies the $B$-condition and $\lambda(X)<\infty$ then the 
restriction map is compact.
\end{proposition}

\begin{proof}
If $R$ satisfies the $B$-condition then $$\lambda^{-2}|\phi(z)|\leq sup_{z\in 
S_R} |\lambda^{-2}(z) \phi(z) | \leq C \|\phi \|_1,$$ hence $|\phi(z)|\leq 
C\| 
\phi\|_1 \lambda^{2}(z),$  by Lebesgue Theorem 
the restriction map is compact.
\end{proof}

Using Theorem \ref{MainTheorem} and Proposition \ref{prop.Bcond.comp} 
we have the following.
\begin{corollary}\label{cor.area}
If $R$ satisfies the $B$-condition and $X$ is an invariant set of 
positive Lebesgue measure with $Area_\lambda(X)<\infty$. If $\mu$ is a non 
zero invariant Beltrami differential, then $m(supp(\mu)\cap X)>0$ if and only 
if $R$ is a flexible Latt\`es map. 
\end{corollary}
In general,  the finiteness of the hyperbolic area of $X$ does not 
imply the finiteness of hyperbolic area of $W$. Generically, it 
could be that the hyperbolic area of $W$ is infinite regardless of the area of 
$X$.  
On the other hand, by Corollary \ref{cor.area}, if $R$ satisfies the 
$B$-condition and the hyperbolic area 
$Area_\lambda(J(R))$ is bounded then $R$ satisfies Sullivan's conjecture. 
However, in this situation, we believe that the following stronger statement 
holds true:

The $Area_\lambda(J(R))<\infty$ if 
and only if either $m(J(R))=0$ or $R$ is postcritically finite. 

In fact, we do not know if the $B$-condition is sufficient on this statement.

Now we consider the more general condition when the restriction map $r_X$ is 
compact. This condition, in some sense,  reflects the geometry of the 
postcritical set.
 
On the product $S_R\times S_R\subset \C^2$  there exist a unique 
function $K(z,\zeta)$ which is characterized by the following conditions.

\begin{enumerate}
 \item $K(\zeta,z)=-\overline{K(z,\zeta)}$
 \item For any $\zeta_0\in S_R$, the function $\phi_{\zeta_0}(z)=K(z,\zeta_0)$
belongs to the intersection $Q(S_R)\cap B(S_R).$ 
 \item If $z_0,\zeta_0$  belong to different components of $S_R$, then 
$K(z_0,\zeta_0)=0.$
\item The operator $P(f)(z)=\int \lambda^{-2}(\zeta) K(z,\zeta)f(\zeta) d\zeta 
d\bar{\zeta}$ from $L_1(S_R)$ to $Q(S_R)$ is a continuous surjective 
projection.
\end{enumerate}

In fact, the function $K(z,\zeta)$ is defined on any planar hyperbolic Riemann 
surface $S$. In particular, when the surface $S$ is the unit disk $\mathbb{D}$ 
the function  $K(z,\zeta)$ has the formula $$K(z,\zeta)=\frac{3}{2}\pi i 
K_\mathbb{D}(z,\zeta)^2$$ where 
$K_\mathbb{D}(z,\zeta)=[\pi(1-z\bar{\zeta})^2]^{-1}$ is the classical Bergman 
Kernel function on the unit disk.  For further details on these facts see for 
example Chapter 3, \S 7 of the book of I. Kra 
\cite{KraBook} .

Now we consider the following function $$\omega(\zeta,z)=\lambda^{-2}(\zeta) 
K(z,\zeta)$$ and $$w(z)=\omega(z,z).$$

The following proposition is a consequence of H\"older inequality and appear as 
Lemma 2 on Ohtake's paper \cite{Ohtakedeform}.

\begin{proposition}\label{prop.bound.comp}
 If $X$ has positive measure and $$\int_X |w|<\infty$$ then the restriction 
$r_X:\phi\mapsto \phi|_X$ from $Q(S_R)$ to $L_1(X)$ is compact.
\end{proposition}
\begin{proof}

We follow arguments of Lemma 2 in \cite{Ohtakedeform}. If $D$ is a 
component of $S_R$, then by H\"older's 
inequality as in Lemma 2 of \cite{Ohtakedeform},   we have that $$|(\phi|_D)(z)| 
\leq 
C|(w|_D)(z)|\int_D |\phi| $$ where the constant $C$ does not depend on $D$.  
Since $S_R$ is a countable union of components, then 

$$|\phi(z)|\leq C |w| \|\phi(z)\|.$$ As $w$ is integrable on $X$ then by 
applying once again the Lebesgue Theorem we complete the proof.
\end{proof}

As a consequence we have:
\begin{corollary}\label{cor.finite}
 If $\int_{J(R)} |w|<\infty$ then $R$ satisfies Sullivan's conjecture.
\end{corollary}

\begin{proof}
 Follows from Theorem \ref{MainTheorem} and Proposition \ref{prop.bound.comp}.
\end{proof}

Remarks: 
\begin{enumerate}
\item If $R$ satisfies the $B$ condition then by Classical results, see the 
comments before Proposition 1 in \cite{Ohtake}, we have that $w(z)\leq C 
\lambda^2(z)$ where 
$C$ does not depend on $z.$ Partially, if $X$ has bounded hyperbolic area then 
$w(z)$ is integrable on $X$, hence the conditions of Proposition 
\ref{prop.bound.comp} implies Proposition \ref{prop.Bcond.comp}.
As it is mentioned in \cite{Ohtake}, the conditions in Proposition 
\ref{prop.Bcond.comp} are strictly weaker than conditions of Proposition 
\ref{prop.bound.comp}.
\item Moreover, by other result of Ohtake (Proposition 3 in 
\cite{Ohtakedeform}) we note 
that in general, the boundedness of the hyperbolic area is not a quasiconformal 
invariant.

\end{enumerate}

In other words, Proposition \ref{prop.Bcond.comp} and Proposition 
\ref{prop.bound.comp} states that if $X$ is completely invariant positive 
measure set and satisfying an integrability condition then $X$ can not support 
extremal 
differentials with  Hamilton-Krushkal degenerated sequences.

Hence, Corollary \ref{cor.area} and Corollary \ref{cor.finite}, in the case 
when $X$ is a completely invariant, derive from results in 
\cite{CMFixed}. Together, the corollaries mean that 
if a map $R$ has an invariant line field which does not allow a  
Hamilton-Krushkal degenerated sequences on $Q(S_R)$, then $R$ is a Latt\`es map 
if and only if the postcritical set has Lebesgue measure zero.

Let $Y_n$ be an exhaustion of $S_R$ by compact subsets such that the Lebesgue 
measure of 
$Y_{n+1}\setminus Y_n$ converge to zero. Let $P_n$ be the sequence of 
restrictions 
$P_n:L_1(S_R)\rightarrow L_1(S_R)$ given by $P_n(f)=\chi_{n} P(f)$ where 
$\chi_{n}$ is the characteristic function on $Y_n$. 
Immediately from the definition we have the following facts:

\begin{enumerate}
 \item For each $n$, the map  $P_n$ is a compact operator.
 \item The limit \[\lim_{n\rightarrow \infty} 
\|P_n(f)-P(f)\|_{L_1(S_R)}\rightarrow 0\] for all $f$ on $L_1(S_R)$.
\end{enumerate}

We have the following Theorem:

\begin{theorem}\label{th.exhaustion}
 Let $\mu\neq 0$ be an extremal invariant Beltrami differential, then the 
following conditions are equivalent:

\begin{itemize}
 \item The map $R$ is a flexible Latt\`es map. 
 \item  There exist an exhaustion of compact sets $Y_n$ as defined above 
such that the following inequality is true:  \[\inf_{n} \|P_n -P\|_{L_1(supp 
(\mu))}<1.\]
\end{itemize}
  
\end{theorem}

\begin{proof}
Assume that $R$ is a flexible Latt\`es map, then $Q(S_R)$ is finitely 
dimensional then the operators $P_n$ converge to $P$ by norm. Hence,
the infimum $\inf_{n} \|P_n -P\|_{L_1(supp(\mu))}=0.$ 

Now, let us assume that $\inf_{n} \|P_n -P\|_{L_1(supp (\mu))}<1$. We show 
that this condition implies that $\mu$ does not accept degenerated 
Hamilton-Krushkal sequences. Indeed, assume that  $\{\phi_n\}$ is a degenerated 
Hamilton-Krushkal sequence for $\mu$. By assumption, there exist $n_0$ such 
that 
\[\sup_{f\in L_1(supp(\mu)),\|f\|=1} \int |P_{n_0}(f)-P(f)|=r<1.\] Since 
$\phi_n$ is degenerated and by the compactness of $P_{n_0}$ we have that 
\[\lim_{j\rightarrow \infty} \|P_{n_0}(\phi_j)\|_{L_1(S_R)}\rightarrow 0.\] 
Hence
\[\|\mu\|_\infty=\lim_{j} \bigg |\int \mu \phi_j\bigg |=\lim_j \bigg 
|\int_{supp(\mu) }\mu 
\phi_j\bigg |\]
\[=\lim_{j} \bigg |\int_{supp(\mu)} \mu(P_{n_0}(\phi_j)-P(\phi_j))\bigg |\]
\[\leq \|\mu\|_\infty \sup_{f\in L_1(supp(\mu)),\|f\|=1} \int 
|P_{n_0}(f)-P(f)|=r \|\mu\|_\infty < \|\mu\|_\infty.\] 
Which is a contradiction. 

Applying the Corollary 1.5 in \cite{EarleLi}, the extremal differential  $\mu$ 
does 
not accept Hamilton{-}Krushkal degenerated sequences if 
and only if there exist $\phi$ in $Q(S_R)$ and a suitable constant $K$ such 
that 
\[v_\mu(\gamma)=K \int \frac{|\phi|}{\phi} \gamma.\]

Hence, for any $\gamma$ in $Q(S_R)$ we have

$$\int \frac{|\phi|}{\phi}R^*(\gamma)=\int\frac{|\phi|}{\phi} \gamma$$ and
$$1=\int\frac{|\phi|}{\phi}R^*(\phi).$$ This implies that 
$$\frac{|R^*(\phi)|}{R^*(\phi)}=\frac{|\phi|}{\phi}$$ but since $\phi$ is 
holomorphic then $\phi$ is a non zero fixed point on $Q(S_R)$. 
Using arguments of the proof of Theorem \ref{MainTechnicalThm} we are 
done. 
\end{proof}

The following Proposition is an illustration of when the conditions of Theorem 
\ref{th.exhaustion} are fulfilled. 

\begin{proposition} If $R$ is a rational map satisfying the $B$ condition.
If $A$ is a measurable subset of $S_R$ so that 
\[\int_A \int_{S_R} |K(z,\zeta)| dz\wedge d\bar{z} \wedge d\zeta \wedge 
d\bar{\zeta}<\infty\]
then for any exhaustion of $S_R$ by compact sets $Y_n$ and operators 
$P_n$ defined as above we have $\lim \|P_n-P\|_{L_1(A)}=0.$ 
\end{proposition}

\begin{proof} Let $Y_n$ be an exhaustion of compact sets as above.  Since 
$K(z,\zeta)$ is absolutely integrable on $A\times S_R$ then 
$$|\chi_{n} K(z,\zeta)|\leq |K(z,\zeta)|$$ and
$\chi_{n}K(z,\zeta)\rightarrow K(z,\zeta)$ pointwise on $A\times S_R.$ By the 
Lebesgue theorem $$\inf \int_{A}\int_{S_R} |K(z,\zeta)-\chi_{n} 
K(z,\zeta)|=0.$$
For all $\phi \in Q(S_R)$, we have $$\|P_n(\phi)-P(\phi)\|_{L_1(A)}$$  

\[\leq \int_A |P_n(\phi)-P(\phi)|\leq \int_{A} \int_{S_R} 
|\lambda^{-2}(\zeta)\phi(\zeta)(K(z,\zeta)-\chi_n K(z,\zeta))|d\zeta dz\]

\[\leq \|\lambda^{-2} \phi\|_\infty \int_{A} \int_{S_R} |K(z,\zeta)-\chi_n 
K(z,\zeta)| d\zeta dz 
\]
which by the $B$-condition we have that the latter is 
\[\leq C \|\phi\|_{L_1(S_R)}\int_{A} \int_{S_R} |K(z,\zeta)-\chi_n 
K(z,\zeta)| d\zeta dz .\] For some constant $C$ which does not depend on 
$\phi$.

Now let $f\in L_1(A)$, since $P$ is a projection then $f=\phi+\omega$ where 
$\phi\in 
Q(S_R),$ $P(\omega)=P_n(\omega)=0$ and $$\|\phi\|_{Q(S_R)}\leq \|P\| 
\|f\|_{L_1(A)}.$$
 Hence $\lim \|P_n-P\|_{L_1(A)}= 0.$

\end{proof}
Finally we characterize a Latt\`es map in terms of the geometry of $Q^*(S_R).$ 
We start with the following definitions.
\begin{definition} 
\begin{enumerate} 
 \item A set $L$ in $Q^*(S_R)$ is called a geodesic ray if $L$ 
is an isometric image of the non negative real numbers $\mathbb{R}_+$. 
 \item Let   $L_1$ and $L_2$ be geodesic rays with parameterizations 
$\psi_1:\mathbb{R}_+\longrightarrow L_1$ and 
$\psi_2:\mathbb{R}_+\longrightarrow L_2$ respectively. The pair of rays $L_1$ 
and $L_2$ are  called equivalent if 
$$\limsup_{t\rightarrow \infty} \| \psi_1(t)-\psi_2(t) \|_T\leq d <\infty$$ for 
some $d.$
\item An element $v$ in $Q^*(S_R)$ is called asymptotically finite if the 
number of equivalence classes of geodesic rays in $Q^*(S_R)$ containing $0$ and 
$v$ is finite.
\end{enumerate}

\end{definition}

Now we characterize rational maps which have asymptotically finite non trivial 
invariant Beltrami differentials.

\begin{theorem}\label{thm. uniq.equiv}
Assume that $S_R$ is connected and let $\mu$ be non trivial an invariant 
Beltrami differential for $R$ supported on $S_R$. Then the functional 
$v_\mu(\phi)=\int \phi\mu$ is asymptotically finite if and only if $R$ is 
Latt\`es.
\end{theorem}

\begin{proof}
 If $R$ is a Latt\`es map then $Q^*(S_R)$ is finitely dimensional and then  
there is only a unique geodesic ray passing through any 
pair of points in $Q^*(S_R)$ see \cite{EarleLi} and \cite{GardLakic}. 

Reciprocally, suppose that the functional $v_\mu$ is asymptotically 
finite. Let us first assume 
that $\|v_\mu \|_{Q^*(S_R)}=\|\mu\|_{L_\infty}.$ Then by Corollary 6.4 in 
\cite{EarleLi}, if $\mu$ accept degenerated 
Hamilton-Krushkal sequences there exist $\C$-linear isometry 
$I:\ell_\infty\rightarrow Q^*(S_R)$ such that if $m$ is the constant 
sequence 
with value $\|\mu\|_\infty$ then $I(m)=v_\mu.$

Now let $\{e_i\}$ be the canonical basis of $\ell_\infty$. Then as in 
\cite{EarleLi}, we define geodesic rays in $\ell_\infty$ as follows:

For any $r\geq \|\mu\|_\infty$ and 

$$\psi_{r,i}(t)=\bigg \{ \begin{array}{l} 
t\cdot m \textnormal{ for } t\leq r. \\
r\cdot m + (t-r)\|\mu\|_\infty e_i \textnormal{ for } t>r.
\end{array}$$

But for all $i_0, r_1,r_2$,  $$\limsup_{t\rightarrow \infty}
\|\psi_{i_0,r_1}(t)-\psi_{i_0,r_2}(t)\|_{\ell_\infty}\leq 
|r_1-r_2|\|\mu\|_\infty.$$

Also for all $i\neq j$ and all $r$ we have 
\[\limsup_{t\rightarrow \infty} \|\psi_{j,r}(t)-\psi_{i,r}(t)\|_\infty\]
\[=\limsup_{t\rightarrow \infty}\|\mu\|_\infty t \|e_i-e_j\|=\infty.\]

But the existence of the isometry $I$ gives a contradiction. 
Hence $\mu$ does not accept Hamilton-Krushkal degenerated sequences.

Now using similar arguments as in the proof of Theorem \ref{th.exhaustion}, we 
complete the proof in the case where $\mu$ is extremal.

Finally we show that if $\mu$ is a non trivial invariant Beltrami differential, 
then there exist an extremal invariant differential $\nu$ such that 
$v_\nu(\gamma)=v_\mu(\gamma)$ for all $\gamma$ in $Q(S_R).$

Indeed, if $\mu$ is not extremal then by the Banach Extension Theorem  and 
Riesz Representation Theorem there exist is another Beltrami differential 
$\alpha$  which is extremal satisfying  
$\|\alpha\|_\infty=\|\mu\|_T<\|\mu\|_\infty$ and 
such that defines the same functional as $\mu$ in $Q(S_R)$. Let $\beta$ be a 
$*$-weak limit of the Ces\`aro averages $C_n(\alpha)=\frac{1}{n} 
\sum_{i=0}^{n-1} 
\alpha(R^i)\frac{\overline{(R^i)'}}{(R^i)'}$, then 
$\beta(R)\frac{\overline{R'}}{R'}=\beta$ and $\|\beta\|_\infty\leq 
\|\alpha\|_\infty.$  Then we claim that $v_\beta=v_\mu$.
Let $\{C_{n_i}(\alpha)\}$ be a sequence of averages $*$-weakly converging to 
$\beta$. For any $\gamma \in Q(S_R)$ we have $$\int \gamma \beta=\lim \int 
C_{n_i}(\alpha)\gamma$$ by duality the previous limit is 
equal to $$\lim \int \alpha \frac{1}{n_i}\sum_{k=0}^{n_i-1} R^{*k}(\gamma)=\lim 
\int \mu \frac{1}{n_i}\sum_{k=0}^{n_i-1} R^{*k}(\gamma)$$ but $\mu$ is an 
invariant differential and again using duality the previous limit becomes 
$$\lim \int \mu \gamma=\int \mu \gamma.$$

Hence for any $\gamma$ in $Q(S_R)$ we have $$\int \beta \gamma= \lim 
\int C_{n_i}(\alpha)\gamma=\int \mu \gamma.$$
Since $\alpha$ is extremal we have 
$\| 
\beta \|_\infty=\| \alpha \|_\infty=\|\mu\|_T$. Thus 
$\beta$ is the desired  extremal invariant differential.  

\end{proof}

To conclude, let us note that the arguments of the theorems in this paper 
work for entire and meromorphic functions in the class
of Eremenko-Lyubich. This is the class of all entire or meromorphic functions 
with finitely many critical and singular values. It is not completely clear 
whether this arguments can be carried on entire or meromorphic functions whose 
asymptotic value set contains a compact set of positive Lebesgue measure.

 \bibliographystyle{amsplain} 

\bibliography{workbib}

\end{document}